\documentclass{amsart}

\usepackage{amssymb}
\usepackage{graphicx}
\usepackage{subcaption}

\theoremstyle{definition}
\newtheorem{theorem}{Theorem}[section]
\newtheorem{lemma}[theorem]{Lemma}
\newtheorem{proposition}[theorem]{Proposition}
\newtheorem{corollary}[theorem]{Corollary}

\newtheorem{example}[theorem]{Example}

\theoremstyle{remark}
\newtheorem{remark}[theorem]{Remark}

\numberwithin{equation}{section}

\newcommand{\M}{\mathcal{M}}
\newcommand{\N}{\mathcal{N}}
\newcommand{\U}{\mathcal{U}}
\newcommand{\R}{\mathbb{R}}
\renewcommand{\H}{\mathbb{H}}
\renewcommand{\S}{\mathbb{S}}

\graphicspath{{images/}}

\begin{document}

\title[Condition Metrics in the Three Classical Spaces]{Condition Metrics in the Three Classical Spaces}

\author[J. G. Criado del Rey]{Juan G. Criado del Rey}
\address{Dpto. de Matem\'aticas, Estad\'istica y Computaci\'on. Facultad de Ciencias. Universidad de Cantabria. Spain.}
\email{juan.gonzalezcr@alumnos.unican.es, jgcriadodelrey@gmail.com}


\date{\today}


\begin{abstract}
Let $(\M,g)$ be a Riemannian manifold and $\N$ a $\mathcal{C}^2$ submanifold without boundary. If we multiply the metric $g$ by the inverse of the squared distance to $\N$, we obtain a new metric structure on $\M\setminus\N$ called the \emph{condition metric}. A question about the behaviour of the geodesics in this new metric arises from the works of Shub and Beltr\'an: is it true that for every geodesic segment in the condition metric its closest point to $\N$ is one of its endpoints? Previous works show that the answer to this question is positive (under some smoothness hypotheses) when $\M$ is the Euclidean space $\mathbb{R}^n$. Here we prove that the answer is also positive for $\M$ being the sphere $\S^n$ and we give a counterexample showing that this property does not hold when $\M$ is the hyperbolic space $\H^n$.
\end{abstract}

\maketitle

\section{Introduction}

In this paper we study the following problem: let $(\M,g)$ be a Riemannian manifold and $\N$ a $\mathcal{C}^2$ submanifold without boundary. We consider a new metric structure $g_\kappa$ on $\M\setminus\N$ obtained by multiplying the metric $g$ by the inverse of the squared distance to $\N$. This is, for a point $x \in \M\setminus\N$,
$$g_{x,\kappa} = d(x,\N)^{-2}g_x,$$

\noindent
where $d(x,\N)$ is the Riemannian distance (w.r.t. $g$) from $x$ to $\N$. We call $g_\kappa$ the \emph{condition metric} on $\M\setminus\N$. The interest of the condition metric comes from the papers of Shub \cite{bezout_vi} and Beltr\'an-Shub \cite{bezout_vii}, where they improve complexity bounds for solving systems of polynomial equations in terms of a certain condition metric on the space $\M$ of systems, with $\N$ being the set of ill-conditioned systems to avoid. Although $(\M\setminus\N,g_\kappa)$ is not always a Riemannian manifold, there is still a sensible way to define the concept of geodesic as a path that locally minimizes the distance. Geodesics in the condition metric try to avoid the submanifold $\N$ because being close to $\N$ increases their length. An interesting question about these geodesics is the following: given a geodesic segment in the condition metric, is it true that the closest point from the segment to $\N$ is one of its endpoints? Sometimes we will refer to this property as `the worst is at the endpoints'.

The function $d(\cdot,\N)$ is not always smooth, but it can be shown that it is always Lipschitz (\cite[Proposition 9]{beltran09}). In this context the condition metric defines a Lipschitz-Riemann structure (in the sense of \cite[Definition 2]{beltran09}) and we have to consider Lipschitz curves on $\M\setminus\N$. For such a curve $\gamma:I\rightarrow\M\setminus\N$ the Rademacher Theorem states that the tangent vector $\dot{\gamma}$ exists almost everywhere, so it makes sense to define the arc length of $\gamma$ w.r.t. $g_\kappa$ by
$$L_\kappa(\gamma) = \int_I \|\dot{\gamma}(t)\|_\kappa dt = \int_I \|\dot{\gamma}(t)\|d(\gamma(t),\N)^{-1}dt.$$

With this definition of arc length, we say that a path $\gamma:[a,b] \rightarrow \mathcal{M}\setminus\mathcal{N}$, parametrized by arc length, is a \emph{minimizing geodesic} in the condition metric if $L_\kappa(\gamma) \leq L_\kappa(c)$ for any Lipschitz curve $c:[a,b] \rightarrow \mathcal{M}\setminus\mathcal{N}$ with $\gamma(a) = c(a)$ and $\gamma(b) = c(b)$. We say that $\gamma$ is a \emph{geodesic} if it is locally a minimizing geodesic.

A sufficient condition for a geodesic $\gamma$ in the condition metric to satisfy that `the worst is at the endpoints' is that the function
\begin{equation}\label{normal_conv}
t \mapsto \frac{1}{d(\gamma(t),\mathcal{N})}
\end{equation}

\noindent
is convex (recall that a function $f:(a,b) \rightarrow \mathbb{R}$ is \emph{convex} if for every $x,y\in(a,b)$ and for every $t \in [0,1]$, $f((1-t)x+ty) \leq (1-t)f(x)+tf(y)$). If we examinate some examples in detail, we rapidly realize that a stronger property is satisfied in many cases: the logarithm of the function \eqref{normal_conv} is also a convex function (this means that \eqref{normal_conv} is a \emph{log-convex} function). We wonder if this is true in general. More precisely, is the real function
\begin{equation}\label{log_conv}
t \mapsto \log\frac{1}{d(\gamma(t),\N)}
\end{equation}
convex for every geodesic $\gamma$ in the condition metric? Answering this question is the main goal of our work and our results about it are summarized in theorems \ref{sn} and \ref{hn}. If \eqref{log_conv} is a convex function for every geodesic $\gamma$ in $g_\kappa$, we will say that the \emph{self-convexity property} is satisfied (maybe the term \emph{self-log-convexity} would be more accurate, but we prefer to use this shorter term). If the distance function $d(\cdot,\N)$ is smooth, then the self-convexity property is equivalent to
\begin{equation}\label{log_conv_equiv}
\frac{d^2}{dt^2}\log\frac{1}{d(\gamma(t),\N)} \geq 0 \quad \equiv \quad \frac{d^2}{dt^2}\log d(\gamma(t),\N) \leq 0,
\end{equation}

\noindent
but if it is not, deciding whether \eqref{log_conv} is a convex function or not is much harder a problem. In many cases we will restrict ourselves to the largest open set $\U \subseteq \M\setminus\N$ such that for every $x \in \U$ the function $d(\cdot,\N)$ is smooth and there is a unique closest point to $x$ in $\N$. If \eqref{log_conv} is a convex function for every geodesic contained in $\U$, we will say that the \emph{smooth self-convexity property} is satisfied. The following result solves the problem for the case $\M = \R^n$:

\begin{theorem}\cite[Theorem 2]{beltran09}\label{rn} The smooth self-convexity property is satisfied for the Euclidean space $\M = \mathbb{R}^n$ endowed with the usual inner product $\langle\cdot,\cdot\rangle$, and $\N$ a complete $\mathcal{C}^2$ submanifold without boundary.\end{theorem}

Our first result is

\begin{theorem}\label{sn} The smooth self-convexity property is satisfied for the sphere $\M = \mathbb{S}^n$ and $\N$ a complete $\mathcal{C}^2$ submanifold without boundary.\end{theorem}

Let us now briefly discuss the importance of Theorem \ref{sn} in the context of the question that originated the study of condition metrics. In \cite{bezout_vi,bezout_vii} the authors noted that studying the condition metric in the set
$$\M_{poly} = \{(f,\zeta)\ |\ f\text{ a polynomial system}, \zeta \in \mathbb{P}(\mathbb{C}^{n+1}), f(\zeta) = 0\},$$

\noindent
where polynomial systems are assumed to be homogeneous of fixed degree in $n+1$ complex variables, with
$$\N_{poly} = \{(f,\zeta)\in\M\ |\ \zeta\text{ is a degenerate zero of }f\},$$

\noindent
could be useful for the design of fast homotopy methods to solve polynomial systems (indeed, the metric used in \cite{bezout_vii} is not exactly the condition metric, but it is closely related to it from \cite[Corollary 6]{beltran_sol_var}). The question of self-convexity turned out to be extremely difficult to analyze in this context, which motivated a theoretical and numerical study \cite{beltran09,beltran12,boito_dedieu} of the linear case
$$\M_{lin} = \{(M,\zeta) \in \mathbb{C}^{n\times(n+1)}\times \mathbb{P}(\mathbb{C}^{n+1})\ |\ M\zeta = 0\},$$

\noindent
(we denote by $\mathbb{C}^{n\times(n+1)}$ the set of $n \times (n+1)$ complex matrices) with
$$\N_{lin} = \{(M,\zeta)\in\M\ |\ \dim\ker M > 1\}.$$

Using quite sophisticated an argument, it was proved in \cite{beltran12} that the self-convexity property holds in $(\M_{lin},\N_{lin})$. The argument considers a stratification of the set $\mathbb{C}^{n\times(n+1)}$ of complex matrices based on the singular value descomposition. For each $u$-uple $(k) = (k_1, ..., k_u)$ of integers with $k_1 + \cdots + k_u = n$, consider the set $\mathcal{P}_{(k)}$ of matrices whose $k_1$ first singular values are equal, whose $k_2$ following singular values are equal, etcetera. That is,
$$\mathcal{P}_{(k)} = \{M\in\mathbb{C}^{n\times(n+1)}\ |\ \text{svd}(M) = (\underbrace{\sigma_1, ..., \sigma_1}_{k_1}, \underbrace{\sigma_2, ..., \sigma_2}_{k_2}, ..., \underbrace{\sigma_u, ..., \sigma_u}_{k_u})\},$$

\noindent
with $\sigma_1 > \sigma_2 > \cdots > \sigma_u > 0$. Also let
$$\mathcal{N}_{(k)} = \{M\in\mathbb{C}^{n\times(n+1)}\ |\ \text{svd}(M) = (\underbrace{\sigma_1, ..., \sigma_1}_{k_1}, \underbrace{\sigma_2, ..., \sigma_2}_{k_2}, ..., \underbrace{\sigma_{u-1}, ..., \sigma_{u-1}}_{k_{u-1}}, \underbrace{0, ..., 0}_{k_u})\}.$$

These sets will play the role of $\M$ and $\N$. It can be shown that $\mathcal{P}_{(k)}$ is a smooth manifold \cite[Proposition 16]{beltran12}. Although in this case $\mathcal{N}_{(k)}$ is not contained in $\mathcal{P}_{(k)}$, $\N_{(k)}$ lies in the boundary of $\mathcal{P}_{(k)}$, so the condition metric in $(\mathcal{P}_{(k)},\N_{(k)})$ can be defined. The distance function is smooth in $\mathcal{P}_{(k)}\setminus\mathcal{N}_{(k)}$ and, surprisingly, the smooth self-convexity property (thus the self-convexity property) holds for each pair $(\mathcal{P}_{(k)},\N_{(k)})$. Then the authors glue all the pieces together and lift the result up to $(\M_{lin},\N_{lin})$, thus proving that the smooth self-convexity property is satisfied in the linear case.

The problem about self-convexity in $(\M_{poly},\N_{poly})$ remains open, but in view of the fact that self-convexity holds for such complicated cases as ($\mathcal{P}_{(k)},\N_{(k)})$, $(\M_{lin},\N_{lin})$ and $\mathbb{R}^n$ together with any $\mathcal{C}^2$ submanifold (Theorem \ref{rn}), one could hope for the existence of a general argument proving that the smooth self-convexity property holds for every pair $(\M,\N)$ under very general assumptions, opening the path to a solution for $(\M_{poly},\N_{poly})$. Theorem \ref{sn} adds another collection of cases to this list, with $\M$ being $\mathbb{S}^n$ and $\N$ any $\mathcal{C}^2$ submanifold.

Despite all this (somehow empirical) evidence, our last theorem shows that smooth self-convexity can fail, even in a very familiar space.

\begin{theorem}\label{hn} If the ambient manifold is the hyperbolic space $\M = \mathbb{H}^n$ and $\N$ is a single point, then for every geodesic $\gamma$ in the condition metric the function $t \mapsto \log\left(\frac{1}{d(\gamma(t),\N)}\right)$ is concave. Moreover, if $\gamma'(t)$ does not point towards the point $\N$, then the function is strictly concave at $t$. Thus in this case the self-convexity property is not satisfied.
\end{theorem}

This result, together with the cases of $\R^n$ (Theorem \ref{rn}) and $\mathbb{S}^n$ (Theorem \ref{sn}), completes the study of the smooth self-convexity property in the three classical spaces.

\section{Some examples}

In this section we will present some examples of condition metrics varying $\M$ and $\N$. From now on, we will denote $d(x,\N)$ simply by $\rho(x)$.

\begin{example} If we take $\M = \mathbb{R}^2$ the Euclidean plane and $\N$ the line $\{y=0\}$, then de distance from a point $(x,y)$ to $\N$ is $\rho(x,y) = y$ and the condition metric reads $g_{(x,y),\kappa} = \frac{1}{y^2}\langle\cdot,\cdot\rangle$. In this case we obtain two copies of the Poincar\'e half space and the function \eqref{log_conv} is convex for every geodesic segment, supporting Theorem \ref{rn}.

\end{example}

\begin{example} Let $\M$ be $\mathbb{R}^2$ as in the previous example and let $\N$ be a single point. For example, let $\N$ be the origin $\N = \{(0,0)\}$ as in Figure \ref{fig:1point}.

\begin{figure}[h]
\begin{subfigure}{.45\textwidth}
  \centering
  \includegraphics[width=1\textwidth]{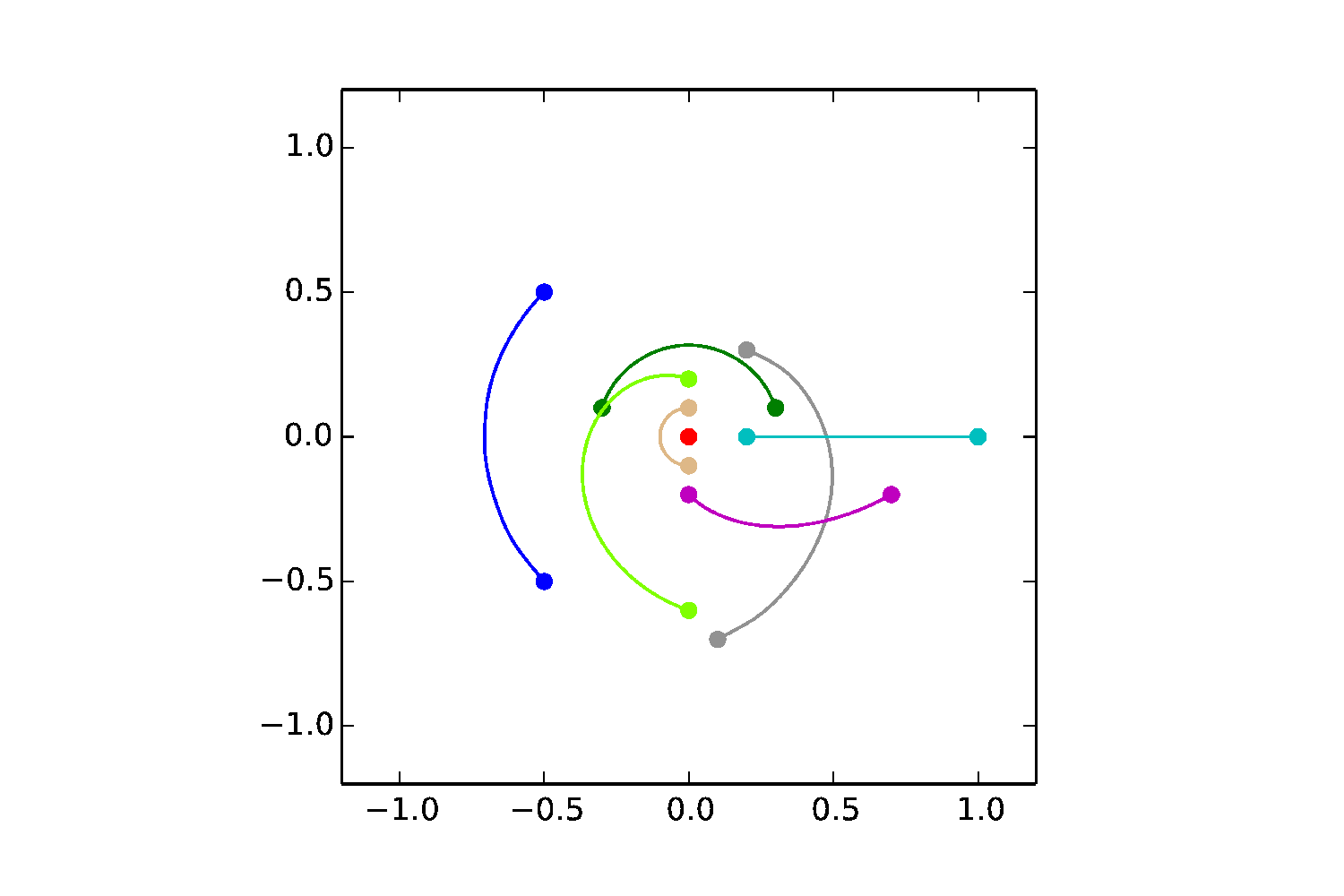}
	\caption{Geodesic segments.}
\end{subfigure}
\begin{subfigure}{.45\textwidth}
  \centering
  \includegraphics[width=1\textwidth]{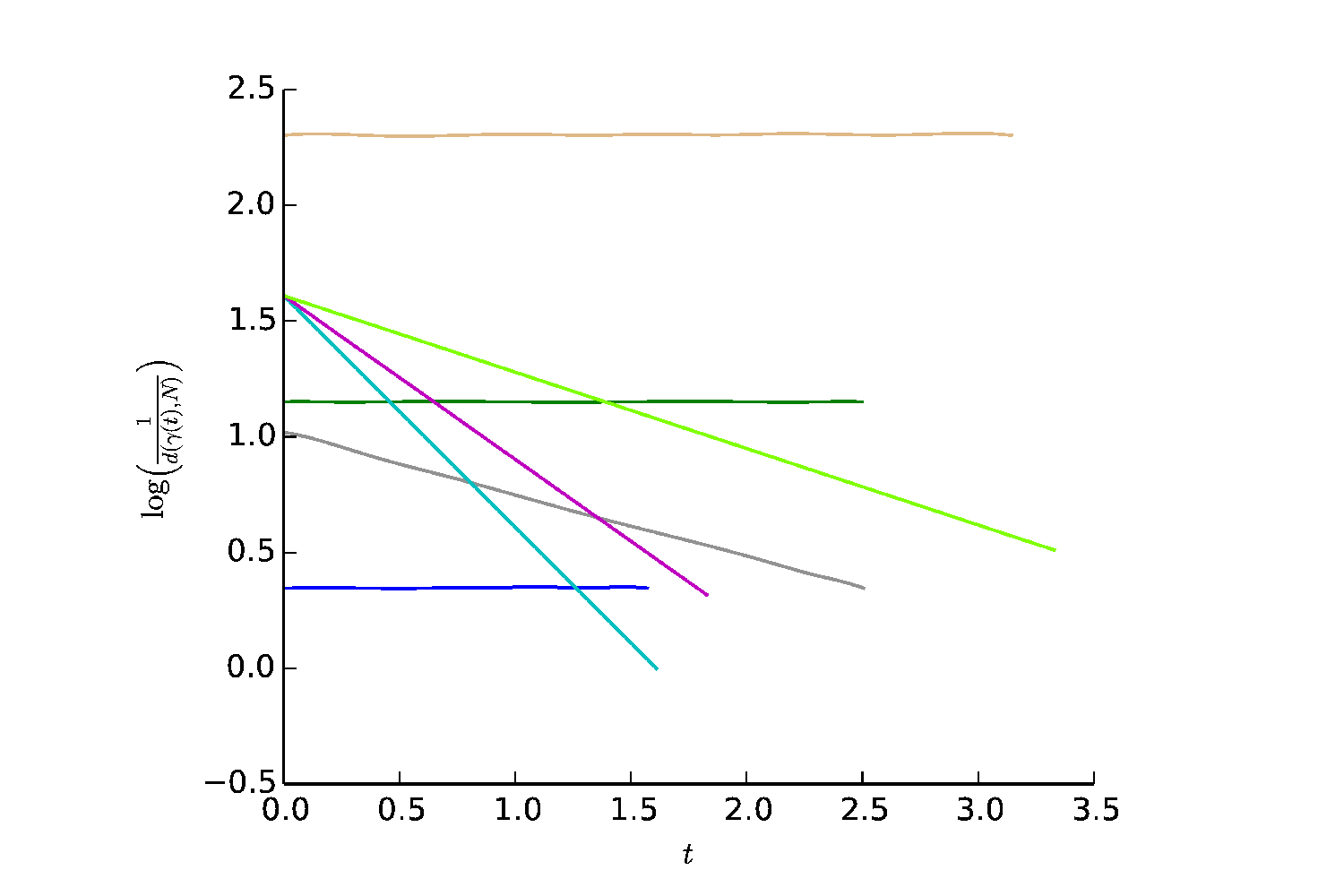}
	\caption{Corresponding \eqref{log_conv} functions.}
\end{subfigure}
\caption{Some geodesic segments in the condition metric when $\M$ is the Euclidean plane $\mathbb{R}^2$ and $\N$ is the red point, $(0,0)$. In this case \eqref{log_conv} functions are affine functions, thus convex.}
\label{fig:1point}
\end{figure}

The condition metric is given by $g_{(x,y),\kappa} = \frac{1}{\|(x,y)\|^2}\langle\cdot,\cdot\rangle$. In this case we are on the hypotheses of Theorem \ref{rn}, so the function \eqref{log_conv} is convex on $\M\setminus\N$. Moreover it can be shown that $(\mathbb{R}^2\setminus\{(0,0)\},g_\kappa)$ is isometric to a cylinder via the isometry $f:\mathbb{R}^2 \rightarrow \mathbb{R}^3$ given by
$$f(x,y) = \left(\frac{x}{\|(x,y)\|},\frac{y}{\|(x,y)\|},\log\|(x,y)\|\right).$$

\end{example}

\begin{example} If we take out two points from the plane, let us say we set $\N = \{(-1,0),(1,0)\}$, then $\rho(x)$ is a piecewise function smooth at every point $(x,y)$ with $x > 0$ or $x < 0$, but it is not smooth on the line $\{x = 0\}$ and for every point in this line there are two closest points to $x$ in $\N$. Theorem \ref{rn} guarantees that \eqref{log_conv} is a convex function for every geodesic segment contained in one of the two semiplanes $\{x>0\}$ or $\{x<0\}$, but it says nothing about those geodesic segments crossing the line $\{x=0\}$. Figure \ref{fig:2points} shows a picture of the situation.

\begin{center}
\begin{figure}[h]
\begin{subfigure}{.45\textwidth}
  \centering
  \includegraphics[width=0.75\textwidth]{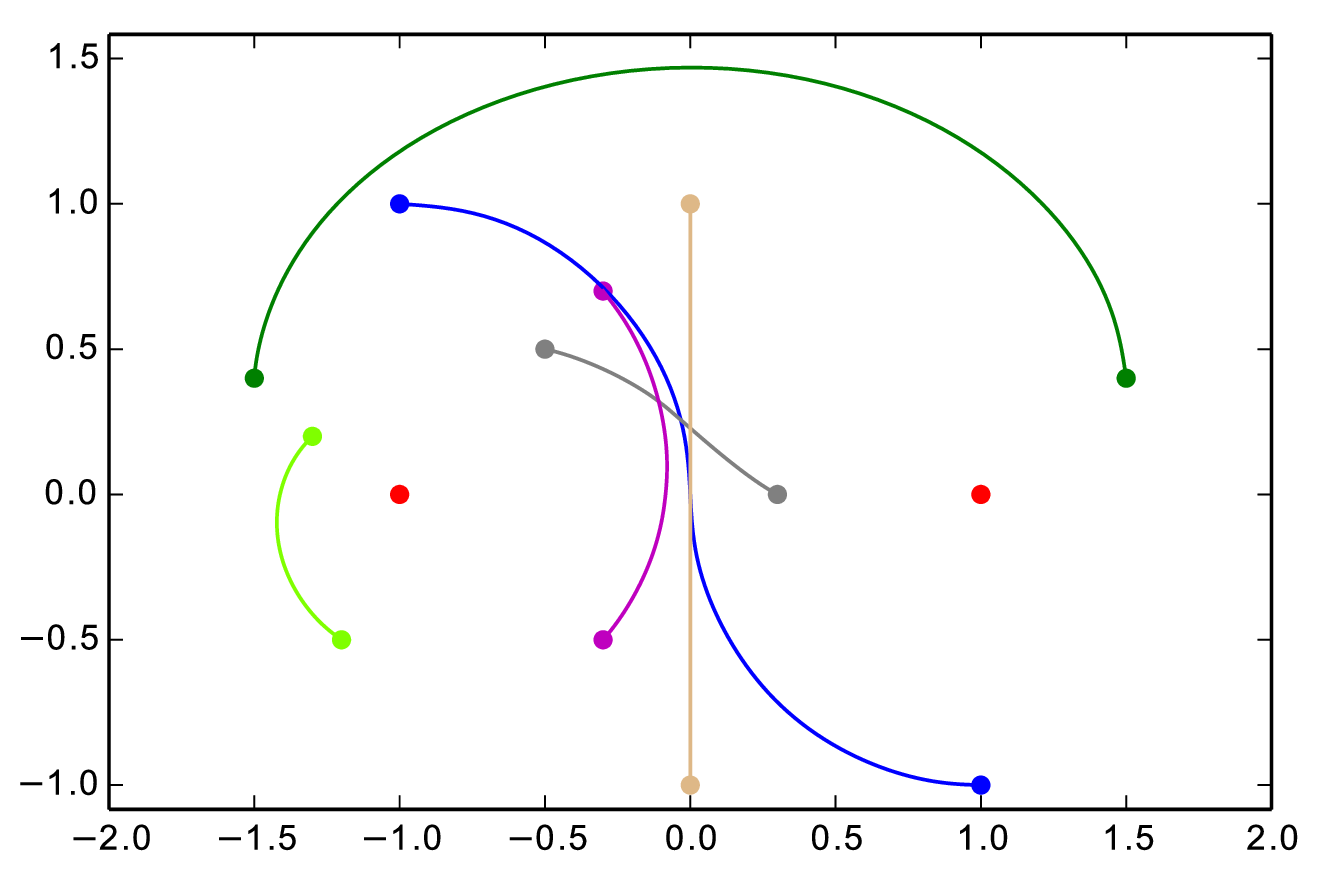}
	\caption{Geodesic segments.}
\end{subfigure}
\begin{subfigure}{.45\textwidth}
  \centering
  \includegraphics[width=1\textwidth]{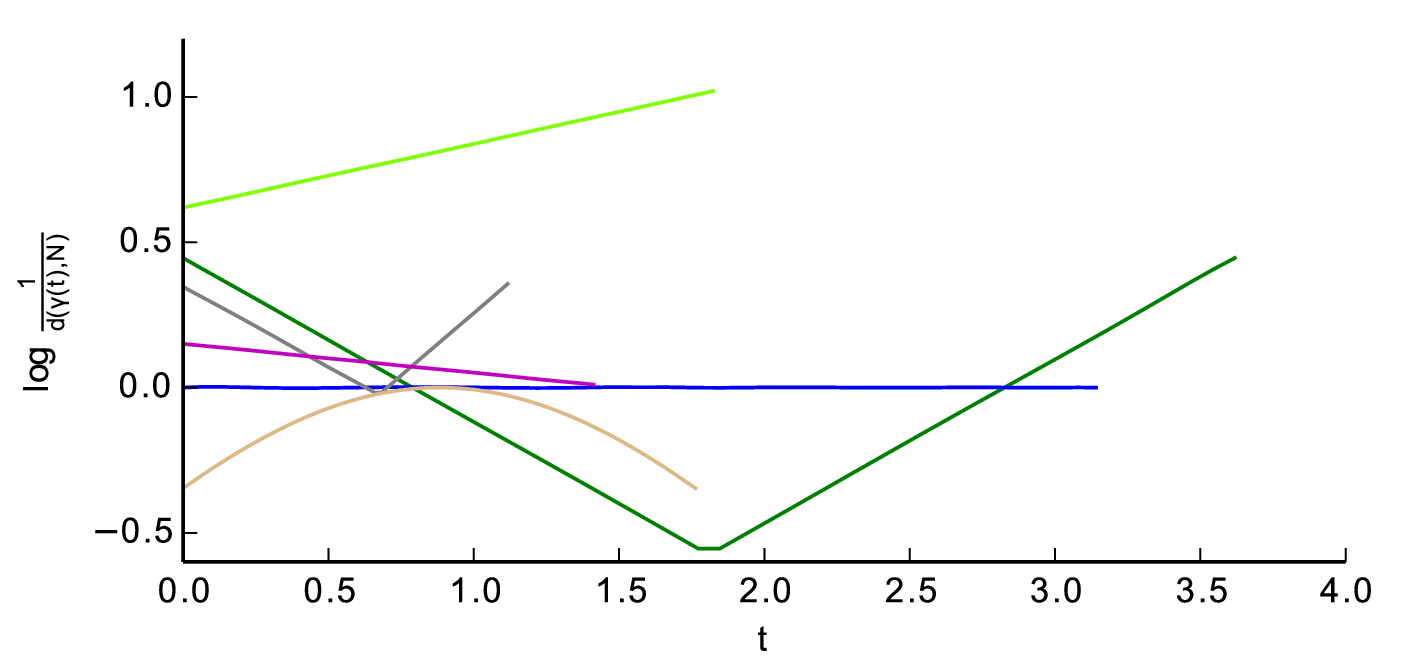}
	\caption{Corresponding \eqref{log_conv} functions.}
\end{subfigure}
\caption{Some geodesic segments in the condition metric when $\M = \R^2$ and $\N$ consists of the two red points.}
\label{fig:2points}
\end{figure}
\end{center}

As we can see, if a geodesic segment which crosses the line $\{x=0\}$ has only one point in this line, then its corresponding \eqref{log_conv} function is convex because both branches of the function are convex and, when crossing the line, the distance function reaches a global maximum, hence \eqref{log_conv} reaches a minimum and is convex (see Lemma \ref{conv_branches}). However, the function \eqref{log_conv} corresponding to the light brown segment, which is entirely contained in the problematic line, is not convex.

The general case for $\N$ being a finite number of points in the plane is determined by the Voronoi diagram of the points. Inside the Voronoi cells \eqref{log_conv} is convex by Theorem \ref{rn}, but we cannot say much about what happens for segments crossing some edges and vertices. \end{example}

\begin{example} If $\M$ is again the plane and $\N$ is a hyperbola, then the situation is very similar to the example above (see Figure \ref{fig:hyperbola}). The function \eqref{log_conv} is convex for every geodesic segment contained in the open set $\U$ where $\rho$ is smooth and there is a single closest point in the hyperbola, but it fails to be convex for the blue segment, which is entirely contained in the $y$ axis: if we have to move from one of the blue dots to the other one, we have to go through the neck of the hyperbola.

\begin{figure}[h]
\begin{subfigure}{.45\textwidth}
  \centering
  \includegraphics[width=0.8\textwidth]{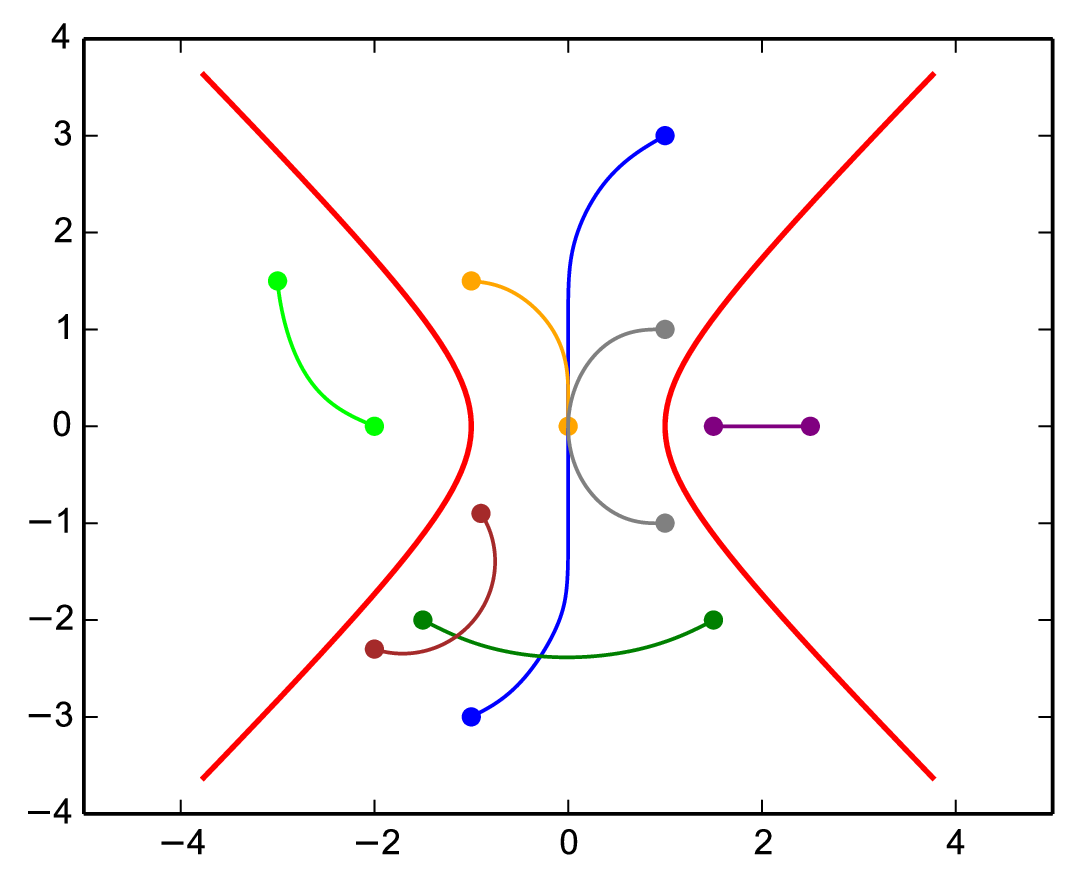}
	\caption{Geodesic segments.}
\end{subfigure}
\begin{subfigure}{.45\textwidth}
  \centering
  \includegraphics[width=1\textwidth]{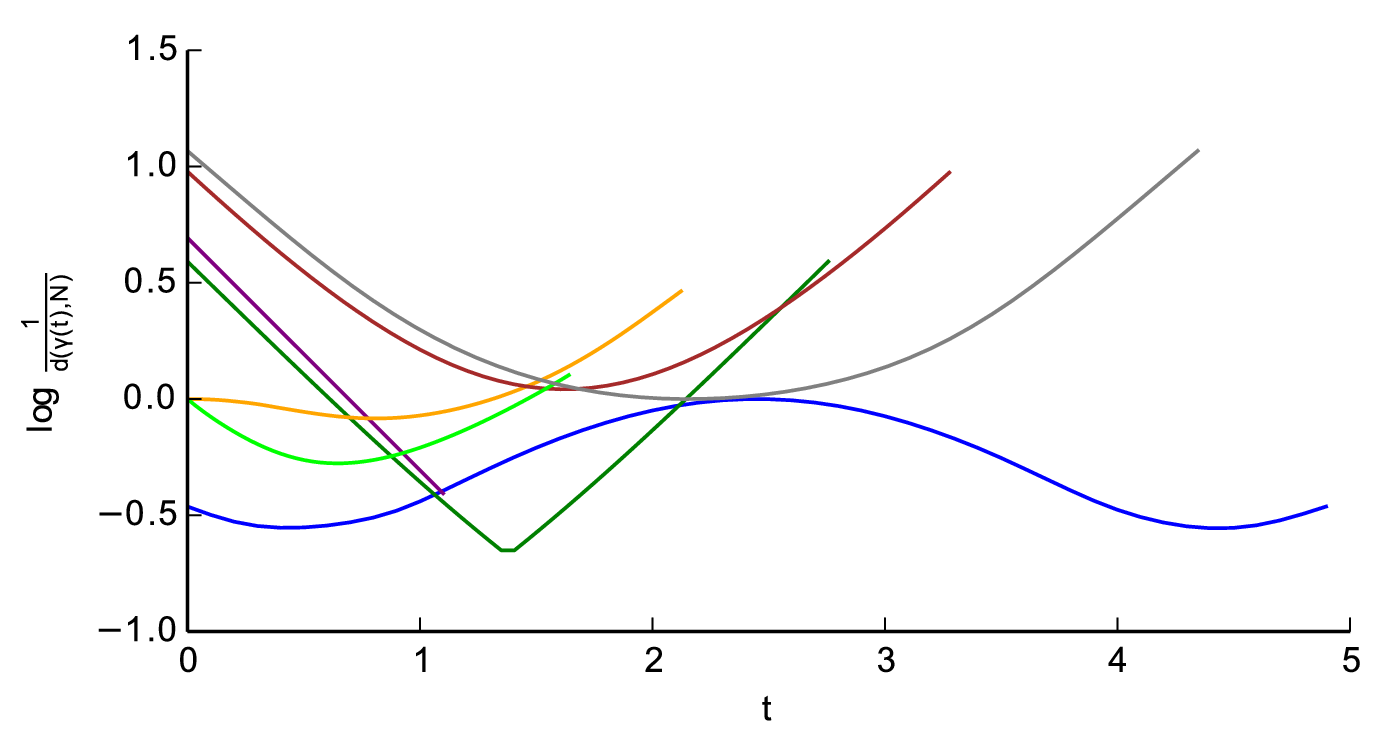}
	\caption{Corresponding \eqref{log_conv} functions.}
\end{subfigure}
\caption{Some geodesic segments in the condition metric when $\M$ is the Euclidean plane $\mathbb{R}^2$ and $\N$ is a hyperbola (in red). Self-convexity fails for the blue segment, which is not contained in $\U$.}
\label{fig:hyperbola}
\end{figure}

\end{example}

\begin{example} Let us move from the Euclidean ambient manifold to the sphere. Let $\M = \S^2$ and $\N$ a single point. For example $\N = \{(0,0,1)\}$ be the north pole $N$, as in Figure \ref{fig:sphere}. In spherical coordinates, the distance from a point $(\theta,\phi)$ to the north pole is simply $\rho(\theta,\phi) = \theta$, hence the local expression for the condition metric in this case is $g_{(\theta,\phi),\kappa} = \frac{1}{\theta^2}g_{(\theta,\phi)}$, where $g$ is the usual metric on the sphere in spherical coordinates. The function $\rho$, defined on $\S^2\setminus\{N\}$, is not smooth at the south pole $S = \{(0,0,-1)\}$, but it is smooth elsewhere, so our main result about self-convexity on the sphere says that \eqref{log_conv} is convex for every geodesic segment contained in $\S^2\setminus\{N,S\}$. However, as a consequence of Lemma \ref{conv_branches}, in this particular case self-convexity also holds at the south pole.
\end{example}

\begin{figure}[h]
\begin{subfigure}{.45\textwidth}
  \centering
  \includegraphics[width=0.6\textwidth]{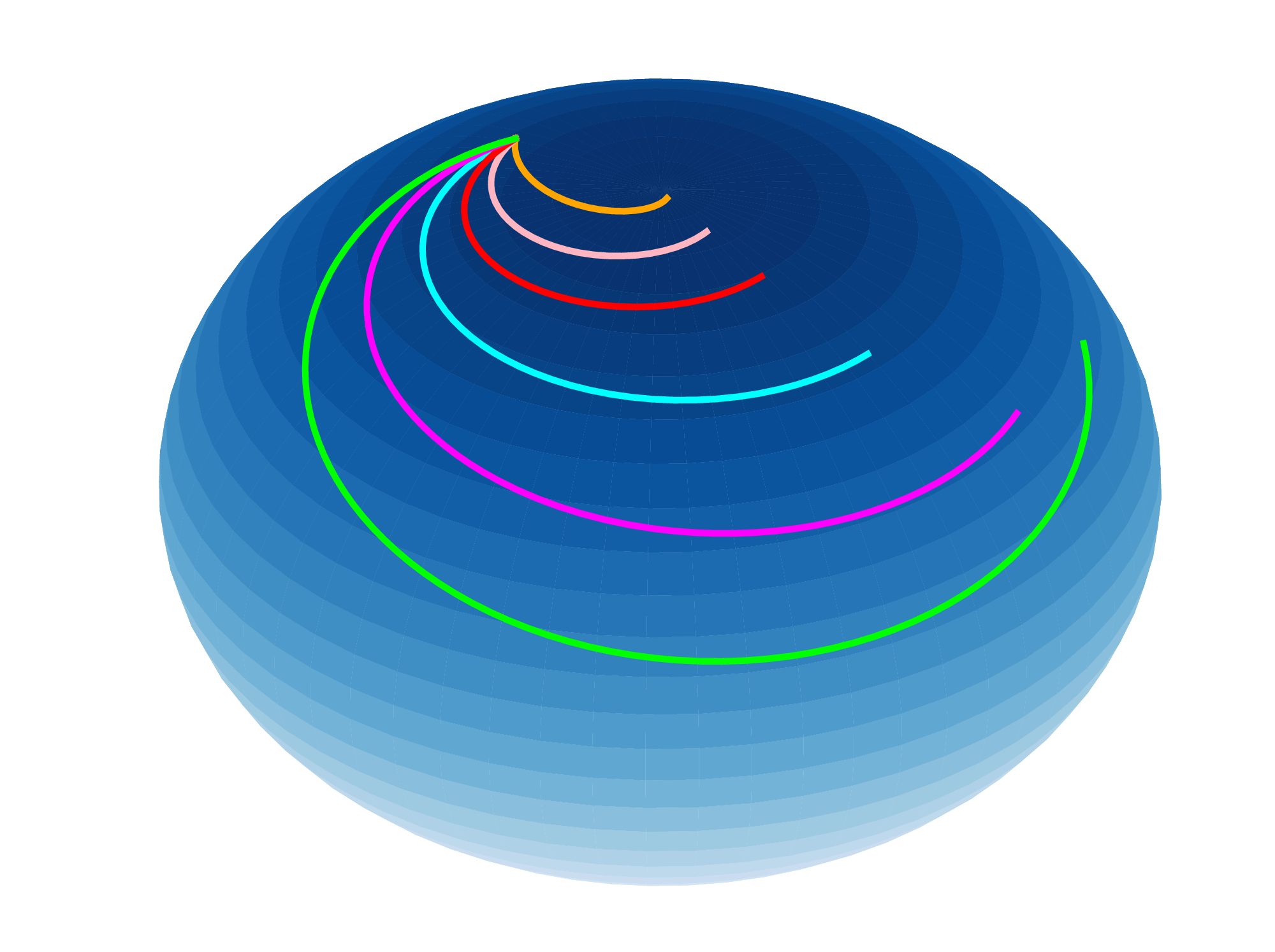}
	\caption{Geodesic segments.}
\end{subfigure}
\begin{subfigure}{.45\textwidth}
  \centering
  \includegraphics[width=0.8\textwidth]{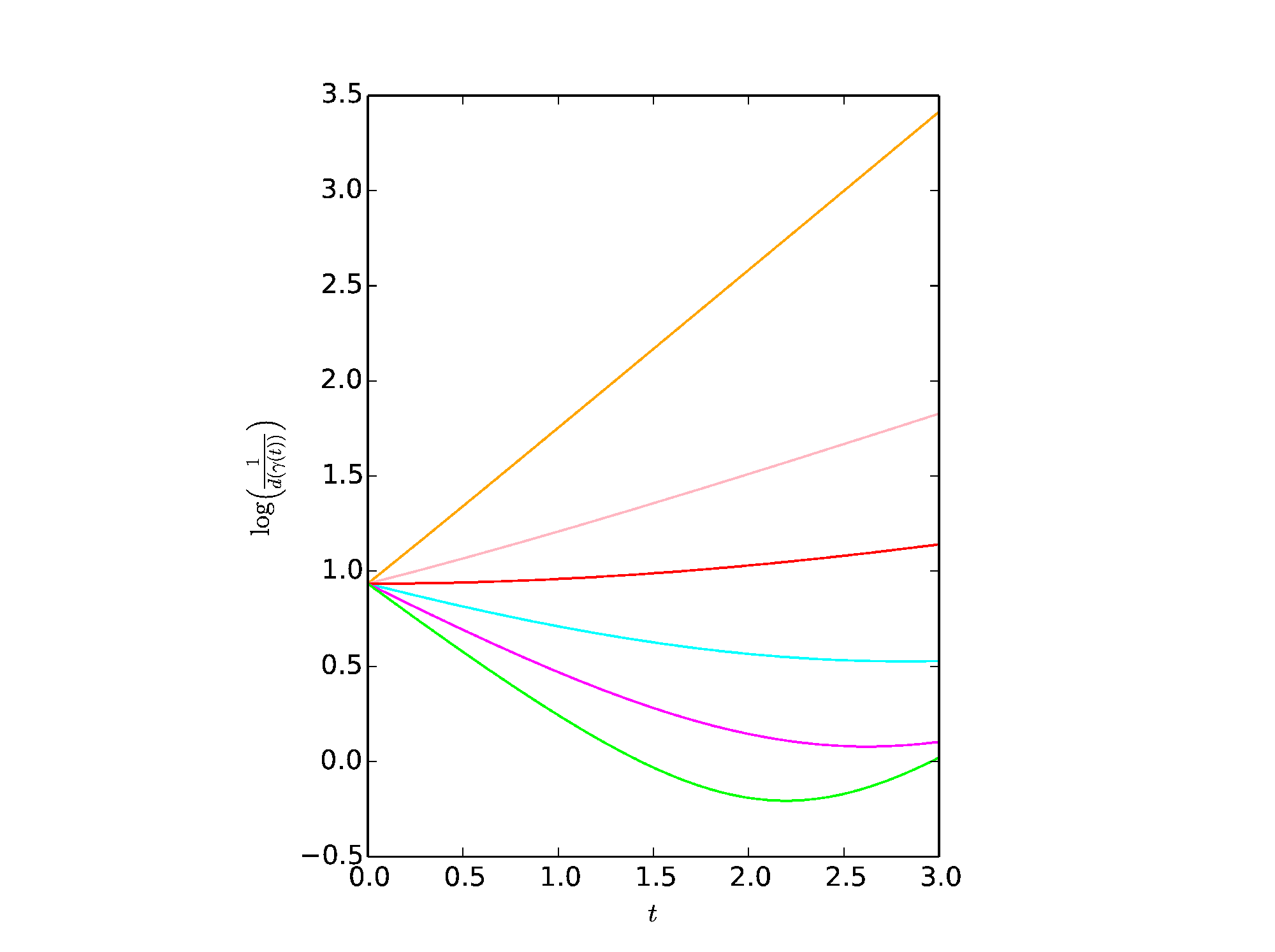}
	\caption{Corresponding \eqref{log_conv} functions.}
\end{subfigure}
\caption{Some geodesic segments in the condition metric when $\M = \S^2$ and $\N$ is a single point, the north pole.}
\label{fig:sphere}
\end{figure}

\begin{example} If $\M$ is the paraboloid given by $z = x^2+y^2$ and $\N$ is the vertex $(0,0,0)$, then the distance from a point $(z\cos\phi,z\sin\phi,z^2)$ to $\N$ is given by the formula $\frac{1}{4}\left(2z\sqrt{4z^2+1}+\arcsin 2z\right)$. In this case the function $\rho$ is smooth everywhere in $\M\setminus\N$ and the numerical experiments suggest that the self-convexity property also holds in this case. Geodesic segments exhibit a curious behaviour: if we throw a geodesic in a direction not opposed to the vertex, it will always eventually fall down towards the vertex describing a spiral (see Figure \ref{fig:paraboloid}).

\begin{figure}[h]
\begin{subfigure}{.45\textwidth}
  \centering
  \includegraphics[width=0.8\textwidth]{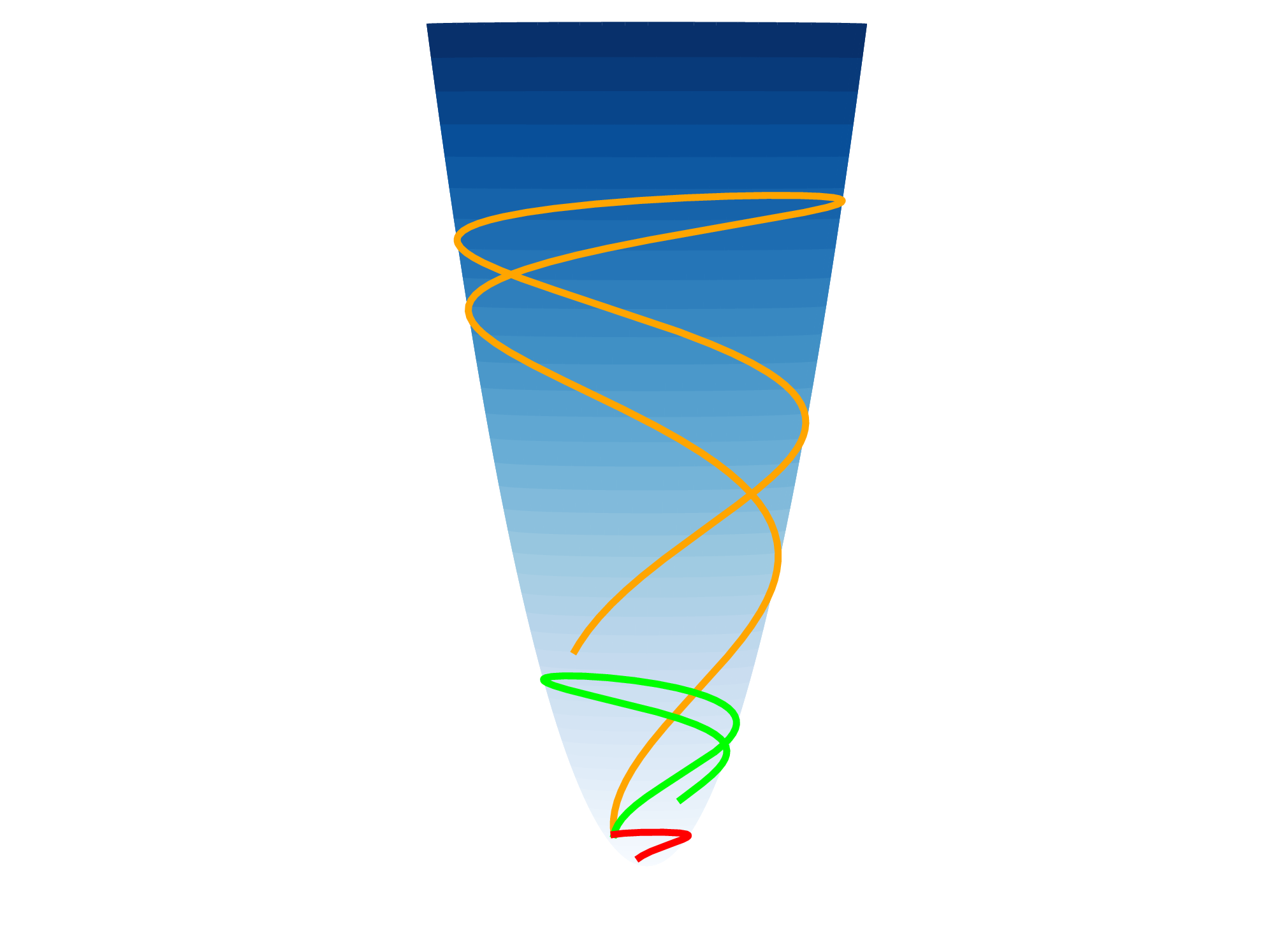}
	\caption{Geodesic segments.}
\end{subfigure}
\begin{subfigure}{.45\textwidth}
  \centering
  \includegraphics[width=0.8\textwidth]{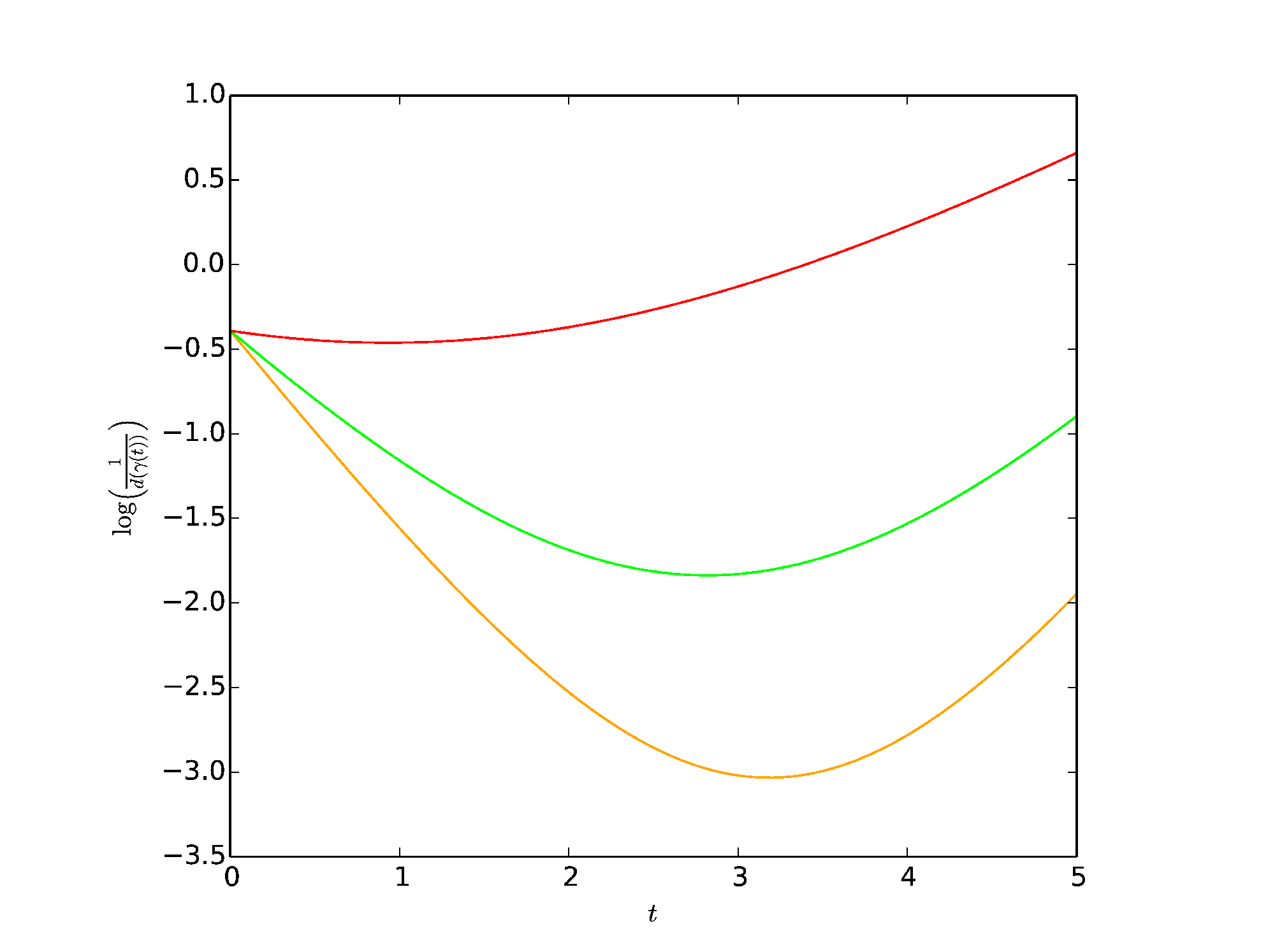}
	\caption{Corresponding \eqref{log_conv} functions.}
\end{subfigure}
\caption{Some geodesic segments in the condition metric when $\M$ is the paraboloid $z = x^2+y^2$ and $\N$ is the vertex of the para\-boloid.}
\label{fig:paraboloid}
\end{figure}

\end{example}

\section{Punctured $\S^n$}

Now let us study the case when $\M$ is the sphere $\S^n$ and $\N$ is a single point, the north pole $\N = \{(1, 0, ..., 0)\}$. The sphere may be parametrized in spherical coordinates as
\begin{align*}
x_1 & = \cos\theta_1,\\
x_2 & = \sin\theta_1\cos\theta_2,\\
x_3 & = \sin\theta_1\sin\theta_2\cos\theta_3,\\
\vdots & \\
x_n & = \sin\theta_1 \cdots \sin\theta_{n-1}\cos\theta_1,\\
x_{n+1} & = \sin\theta_1 \cdots \sin\theta_{n-1}\sin\theta_n,
\end{align*}

\noindent
where $\theta_1, ..., \theta_{n-1}\in (0,\pi)$ and $\theta_n \in (-\pi,\pi)$. The metric tensor with this parametrization is the diagonal matrix
$$g_{\theta} = \sum_{i=1}^n\left(\prod_{j=1}^{i-1}\sin^2\theta_j\right)d\theta_i^2$$

\noindent
and the distance from a point $\theta = (\theta_1, ..., \theta_n)$ to the north pole is $\theta_1$. This yields the condition metric $g_{\theta,\kappa} = \theta_1^{-2}g_\theta$. After the computation of the Christoffel symbols (see, for example, \cite{docarmo}) $\Gamma_{ij}^1$, we obtain
$$\Gamma_{11}^1 = -\frac{1}{\theta_1}, \qquad \Gamma_{12}^1 = 0, \qquad \Gamma_{22}^1 = -\frac{\theta_1\sin\theta_1\cos\theta_1-\sin^2\theta_1}{\theta_1},$$

\noindent
and, for every $j > 2$,
$$\Gamma_{1j}^1 = 0, \qquad \Gamma_{jj}^1 = -\frac{\theta_1\sin\theta_1\cos\theta_1-\sin^2\theta_1}{\theta_1}\prod_{r=2}^{j-1}\sin^2\theta_r.$$

The remaining $\Gamma_{ij}^1$ are zero. With the Christoffel symbols we obtain the first of the geodesic equations, which is the only one that we will need.

\begin{eqnarray}
\ddot{\theta}_1 -\frac{1}{\theta_1}\dot{\theta}_1^2-\frac{\theta_1\sin\theta_1\cos\theta_1-\sin^2\theta_1}{\theta_1}\dot{\theta}_2^2\nonumber &\\
-\sum_{j=3}^n\left(\frac{\theta_1\sin\theta_1\cos\theta_1-\sin^2\theta_1}{\theta_1}\prod_{r=2}^{j-1}\sin^2\theta_r\right)\dot{\theta}_j^2 &= 0.\label{geod_sn_cond}
\end{eqnarray}

\begin{proposition}\label{punctured_sn_smooth} For $\M = \S^n$ and $\N$ a single point, the smooth self-convexity property holds.\end{proposition}

\begin{proof} Let $\gamma$ be a geodesic, so the distance function from $\gamma$ to the north pole is $\gamma_1$. Replacing $\gamma$ in \eqref{geod_sn_cond} and multiplying this equation by $\gamma_1$, we obtain
\begin{align}\gamma_1''\gamma_1 -\gamma_1'^2 &= (\gamma_1\sin\gamma_1\cos\gamma_1-\sin^2\gamma_1)\gamma_2'^2\nonumber\\
&+ \sum_{j=3}^n\left((\gamma_1\sin\gamma_1\cos\gamma_1-\sin^2\gamma_1)\prod_{r=2}^{j-1}\sin^2\gamma_r\right)\gamma_j'^2.\label{geod_sn_cond_2}\end{align}

The real function $x \mapsto x\sin x\cos x-\sin^2x$ is negative for every $x \in (0,\pi)$, so the left hand side of \eqref{geod_sn_cond_2} is always negative. Now note that
$$\frac{d^2}{dt^2}\log\rho(\gamma(t)) = \frac{d^2}{dt^2}\log\gamma_1(t) = \frac{\gamma_1''\gamma_1-\gamma_1'^2}{\gamma_1^2} \leq 0,$$

\noindent
satisfying \eqref{log_conv_equiv}.
\end{proof}

Although it is not clear in spherical coordinates, the distance function is not smooth at the south pole $(-1, 0, ..., 0)$, but the self-convexity property also holds here. In order to prove this fact, we will need the following result.

\begin{lemma}\label{conv_branches} Let $f:(a,b) \rightarrow \R$ a continuous function that reaches a global minimum at $c \in (a,b)$. If both branches $f_1 = f\big|_{(a,c)}$ and $f_2 = f\big|_{(c,b)}$ are convex, then $f$ is convex.\end{lemma}

The proof is left as an exercise to the reader.

\begin{corollary}\label{punctured_sn} For $\M = \S^n$ and $\N$ a single point, the self-convexity property holds.\end{corollary}

\begin{proof} Proposition \ref{punctured_sn_smooth} guarantees that the self-convexity property holds for every geodesic contained in $\U$. Let $\gamma:(-\varepsilon,\varepsilon) \rightarrow \S^n$ be a geodesic across the south pole, with $\gamma(0) = (-1,0, ..., 0) = S$. Since $\gamma$ is locally minimizing, we may suppose that $\gamma(t) \neq S$ if $t \neq 0$, so $0$ is a global minimum for the function $t \mapsto \log\frac{1}{\rho(\gamma(t))}$. Restricting this function to $(-\varepsilon,0)$ and $(0,\varepsilon)$ yields two convex branches by Proposition \ref{punctured_sn_smooth} and the whole function is convex by Lemma \ref{conv_branches}.\end{proof}

\section{Preliminary results}

Before proving Theorem \ref{sn} we will present some technical results that will be useful when doing calculations. We will denote by $K(x)$ the (unique) closest point of $\N$ to a point $x \in \U$. We have the following facts about $K$ and $\rho$ (see also Foote \cite{foote}, Li and Nirenberg \cite{nirenberg}):

\begin{proposition}\cite[Proposition 9]{beltran09} The distance function $\rho$ is $\mathcal{C}^2$ on $\U$ and the function $K$ is $\mathcal{C}^1$ on $\U$.
\end{proposition}

\begin{lemma}\label{xminuskx} The vector $x-K(x)$ is orthogonal to $T_{K(x)}\N$.\end{lemma}

\begin{proof} Let $x,y \in \S^2 \subset \mathbb{R}^3$ be two points. Then the spherical distance between $x$ and $y$ is $d_{\S^2}(x,y) = 2\arcsin\left(\frac{\|x-y\|}{2}\right)$. Let us fix $x$ and consider the function $\delta:\N\rightarrow \mathbb{R}$ given by
$$\delta(y) = d_{\S^2}(x,y) = 2\arcsin\left(\frac{\|x-y\|}{2}\right).$$

This function reaches a minimum at $y = K(x)$, hence $D\delta_{K(x)} \equiv 0$. Let $\dot{x}$ be a tangent vector to $\N$ at the point $K(x)$ and let $c$ be a smooth curve with $c(0) = K(x)$ and $c'(0) = \dot{x}$. Then
$$\frac{d}{dt}\delta(c(t)) = -\left(1-\frac{\|x-c(t)\|^2}{4}\right)^{-1/2}\frac{\langle x-c(t),c'(t)\rangle}{\|x-c(t)\|},$$

\noindent
(note that $\frac{d}{dt}\delta(c(t))$ is well-defined because we are on $\U$) and so
$$0 = D\delta_{K(x)}\dot{x} = \frac{d}{dt}\Big|_{t=0}\delta(c(t)) = -\left(1-\frac{\|x-K(x)\|^2}{4}\right)^{-1/2}\frac{\langle x-K(x),\dot{x}\rangle}{\|x-K(x)\|}.$$

The product above is $0$ if and only if $\langle x-K(x),\dot{x}\rangle = 0$.\end{proof}

\begin{remark}\label{shortcutremark} Lemma \ref{xminuskx} and the fact that $\langle c,c'\rangle = 0$ for every curve $c:I \rightarrow \S^2$, give us a shortcut that we will use many times in calculations:

$$\langle c(t)-K(c(t)),c'(t)-DK_{c(t)}c'(t)\rangle$$
$$= \langle c(t)-K(c(t)),-DK_{c(t)}c'(t)\rangle + \langle c(t),c'(t)\rangle + \langle -K(c(t)),c'(t)\rangle$$
\begin{equation}\label{shortcut} = \langle -K(c(t)),c'(t)\rangle.\end{equation}
\end{remark}

We slightly rephrase \cite[Proposition 3]{beltran09} here.

\begin{proposition}\label{desig} Let $\gamma(t)$ be a geodesic in the condition metric with $\gamma(0) = x \in \U$ and $\gamma'(0) = \dot{x}$. Then the sign of the second derivative of the function \eqref{log_conv} is the same as the sign of the following quantity:
$$\|\dot{x}\|^2\|D\rho_x\|^2-(D\rho_x\dot{x})^2-\rho(x)D^2\rho_x(\dot{x},\dot{x}),$$

\noindent
where the norms and the second covariant derivative are taken with respect to the original metric on $\M$.

In particular, the smooth self-convexity property is satisfied if and only if the quantity above is nonnegative for every $x \in \U$ and $\dot{x} \in T_x\U$.
\end{proposition}

\begin{remark} For every $x \in \S^n \subset \mathbb{R}^{n+1}$ and $\dot{x}\in T_x\S^n$, the unique maximal geodesic $\gamma$ with $\gamma(0) = x$ and $\gamma'(0) = \dot{x}$ is given by $\gamma(t) = \cos(\|\dot{x}\|t)x+\frac{1}{\|\dot{x}\|}\sin(\|\dot{x}\|t)\dot{x}$, so one can check that for any such a geodesic,
\begin{equation}\label{geodsphere}\gamma''(0) = -\|\dot{x}\|^2x\end{equation}
\end{remark}

In order to apply Proposition \ref{desig} we need to compute the derivatives of $\rho$ with respect to the original metric on the sphere. Let $x \in \U$ and $\dot{x} \in T_x\U$, and let $c:I \rightarrow \S^n$ be a curve with $c(0) = x$ and $c'(0) = \dot{x}$. Then $D\rho_x\dot{x} = \frac{d}{dt}\big|_{t=0}\rho(c(t))$ and
\begin{align*}
\frac{d}{dt}\rho(c(t)) & = \frac{d}{dt}2\arcsin\left(\frac{\|c(t)-K(c(t))\|}{2}\right) \\
& = 2\left(1-\frac{\|c(t)-K(c(t))\|^2}{4}\right)^{-1/2}\frac{1}{2}\frac{\langle c(t)-K(c(t)),c'(t)-DK_{c(t)}c'(t)\rangle}{\|c(t)-K(c(t))\|} \\
& = \left(1-\frac{\|c(t)-K(c(t))\|^2}{4}\right)^{-1/2}\frac{\langle -K(c(t)),c'(t)\rangle}{\|c(t)-K(c(t))\|},
\end{align*}
\noindent
where we have used \eqref{shortcut} for the last equality. Then,

\begin{lemma}\label{firstderlemma}For every $x \in \U\subseteq \S^n$ and $\dot{x} \in T_x\U$, we have that
\begin{equation}\label{firstder}D\rho_x\dot{x} = -\left(1-\frac{\|x-K(x)\|^2}{4}\right)^{-1/2}\frac{\langle K(x),\dot{x}\rangle}{\|x-K(x)\|}.\end{equation}
\end{lemma}

Now let us compute the second covariant derivative $D^2\rho_x(\dot{x},\dot{x})$ with respect to the original metric on the sphere. Let $\gamma:I \rightarrow \U$ be a geodesic with $\gamma(0) = x$ and $\gamma'(0) = \dot{x}$. We have that $D^2\rho_x(\dot{x},\dot{x}) = \frac{d^2}{dt^2}\big|_{t=0}\rho(\gamma(t))$ and
$$\frac{d^2}{dt^2}\rho(\gamma(t)) = \frac{d}{dt}\left[-\left(1-\frac{\|\gamma(t)-K(\gamma(t))\|^2}{4}\right)^{-1/2}\frac{\langle K(\gamma(t)),\gamma'(t)\rangle}{\|\gamma(t)-K(\gamma(t))\|}\right].$$

Consider the functions
$$p(t) = -\left(1-\frac{\|\gamma(t)-K(\gamma(t))\|^2}{4}\right)^{-1/2}, \qquad q(t) = \frac{\langle K(\gamma(t)),\gamma'(t)\rangle}{\|\gamma(t)-K(\gamma(t))\|},$$

\noindent
so that $\frac{d}{dt}\rho(\gamma(t)) = p(t)q(t)$. Then
\begin{align*}
\frac{d}{dt}p(t) & = \frac{1}{2}\left(1-\frac{\|\gamma(t)-K(\gamma(t))\|^2}{4}\right)^{-3/2}\left(-\frac{1}{2}\langle \gamma(t)-K(\gamma(t)),\gamma'(t)-DK_{\gamma(t)}\gamma'(t)\rangle\right)\\
& = -\frac{1}{4}\left( 1-\frac{\|\gamma(t)-K(\gamma(t))\|^2}{4}\right)^{-3/2}\langle -K(\gamma(t)),\gamma'(t)\rangle,
\end{align*}

\noindent
where, again, we have used \eqref{shortcut}. Hence
\begin{equation}\label{funcp}
\frac{d}{dt}\Big|_{t=0}p(t) = \frac{1}{4}\left(1-\frac{\|x-K(x)\|^2}{4}\right)^{-3/2}\langle K(x),\dot{x}\rangle.\end{equation}

\noindent
Now
\begin{align*}
\frac{d}{dt}q(t) & = \frac{\left(\frac{d}{dt}\langle K(\gamma(t)),\gamma'(t)\rangle\right)\|\gamma(t)-K(\gamma(t))\|}{\|\gamma(t)-K(\gamma(t))\|^2} \\
 & - \frac{\langle K(\gamma(t)),\gamma'(t)\rangle\frac{\langle \gamma(t)-K(\gamma(t)),\gamma'(t)-DK_{\gamma(t)}\gamma'(t)\rangle}{\|\gamma(t)-K(\gamma(t))\|}}{\|\gamma(t)-K(\gamma(t))\|^2}\\
 & = \frac{\left(\langle DK_{\gamma(t)}\gamma'(t)\rangle+\langle K(\gamma(t)),\gamma''(t)\rangle\right)\|\gamma(t)-K(\gamma(t))\|^2+\langle K(\gamma(t)),\gamma'(t)\rangle^2}{\|\gamma(t)-K(\gamma(t))\|^3}.
\end{align*}

\noindent
This yields
\begin{equation}\label{funcq}
\frac{d}{dt}\Big|_{t=0}q(t) = \frac{\left(\langle DK_x\dot{x},\dot{x}\rangle+\langle K(x),\ddot{x}\rangle\right)\|x-K(x)\|^2+\langle K(x),\dot{x}\rangle^2}{\|x-K(x)\|^3}.\end{equation}

\noindent
Using \eqref{funcp} and \eqref{funcq},
\begin{align*}
\frac{d^2}{dt^2}\Big|_{t=0}\rho(\gamma(t)) & = q(0)\frac{d}{dt}\Big|_{t=0}p(t)+p(0)\frac{d}{dt}\Big|_{t=0}q(t)\\
& = \frac{\langle K(x),\dot{x}\rangle}{\|x-K(x)\|}\frac{1}{4}\left(1-\frac{\|x-K(x)\|^2}{4}\right)^{-3/2}\langle K(x),\dot{x}\rangle\\
& - \left(1-\frac{\|x-K(x)\|^2}{4}\right)^{-1/2}\frac{\left(\langle DK_x\dot{x},\dot{x}\rangle+\langle K(x),\ddot{x}\rangle\right)\|x-K(x)\|^2+\langle K(x),\dot{x}\rangle^2}{\|x-K(x)\|^3}\\
& = \frac{1}{\|x-K(x)\|}\left(1-\frac{\|x-K(x)\|^2}{4}\right)^{-1/2} \cdot \\
& \left[\frac{1}{4}\langle K(x),\dot{x}\rangle^2\left(1-\frac{\|x-K(x)\|^2}{4}\right)^{-1}\right.\\
& \left.- \left(\langle DK_x\dot{x},\dot{x}\rangle + \langle K(x),\ddot{x}\rangle + \frac{\langle K(x),\dot{x}\rangle^2}{\|x-K(x)\|^2}\right)\right].
\end{align*}

\noindent
Finally, we use the fact that $\gamma$ is a geodesic w.r.t. the original metric on the sphere and, by \eqref{geodsphere},
$$\langle K(x),\ddot{x}\rangle = \langle K(x),-\|\dot{x}\|^2 x\rangle = -\|\dot{x}\|^2\langle K(x),x\rangle.$$

\noindent
Putting all these computations together,

\begin{lemma}\label{secondderlemma}For every $x \in \U\subseteq \S^n$ and $\dot{x} \in T_x\U$,
\begin{align*}\label{secondder}
D^2\rho_x(\dot{x},\dot{x}) & = \frac{1}{\|x-K(x)\|}\left(1-\frac{\|x-K(x)\|^2}{4}\right)^{-1/2}\cdot\\
& \left[\frac{1}{4}\langle K(x),\dot{x}\rangle^2\left(1-\frac{\|x-K(x)\|^2}{4}\right)^{-1}\right.\\
&\left.-\left(\langle DK_x\dot{x},\dot{x}\rangle - \|\dot{x}\|^2\langle K(x),x\rangle + \frac{\langle K(x),\dot{x}\rangle^2}{\|x-K(x)\|^2}\right)\right].
\end{align*}
\end{lemma}

\begin{lemma}\label{lemita} For every $x \in \U\subseteq \S^n$ and $\dot{x} \in T_x\U$ we have that $\langle DK_x\dot{x},\dot{x}\rangle \geq 0$.\end{lemma}

\begin{proof} Let $c:I \rightarrow \U$ be a curve with $c(0) = x$ and $c'(0) = \dot{x}$. Let $h > 0$ be a positive real number. We will denote by $o(h)$ a generic function satisfying
$$\lim_{h \rightarrow 0}\frac{o(h)}{h} = 0.$$

Applying Taylor's Theorem, we define
$$\tilde{x} = c(h) = c(0) + h c'(0) + o(h) = x + h\dot{x} + o(h).$$

\noindent
We have that
$$K(\tilde{x}) = K(x+h x + o(h)) = K(x) + h DK_x\dot{x}+o(h).$$

\noindent
Now $K(\tilde{x})$ minimizes the distance from $\tilde{x}$ to $N$, so
$$d_{\S^n}(\tilde{x},K(\tilde{x})) \leq d_{\S^n}(\tilde{x},K(x))$$

\noindent
and because $\arcsin$ is an increasing function,
$$\|\tilde{x}-K(\tilde{x})\|^2 \leq \|\tilde{x}-K(x)\|^2.$$

\noindent
Let us compute the quantity on the left.
\begin{align*}
\|\tilde{x}-K(\tilde{x})\|^2 & = \langle \tilde{x}-K(x),\tilde{x}-K(x)\rangle-2\langle \tilde{x}-K(x),h DK_x\dot{x}\rangle + o(h)\\
& = \|\tilde{x}-K(x)\|^2-2\langle \tilde{x}-K(x),h DK_x\dot{x}\rangle + o(h).
\end{align*}

\noindent
Then, necessarily, $2\langle \tilde{x}-K(x),h DK_x\dot{x}\rangle+o(h) \geq 0$. Dividing by $2h$ and as $h$ tends to $0$, $\langle \tilde{x}-K(x),DK_x\dot{x}\rangle \geq 0$. But this quantity is
\begin{align*}
\langle \tilde{x}-K(x),DK_x\dot{x}\rangle & = \langle x+h \dot{x}-K(x)+o(h),DK_x\dot{x}\rangle \\
& = \langle x-K(x),DK_x\dot{x}\rangle + h\langle \dot{x},DK_x\dot{x}\rangle + o(h)\\
& = h \langle \dot{x},DK_x\dot{x}\rangle + o(h),
\end{align*}

\noindent
where the last equality follows from Lemma \ref{xminuskx}. Again, dividing by $h$ and as $h$ tends to $0$, the statement follows.\end{proof}

Now let us compute the operator norm of $D\rho_x$.

\begin{lemma}\label{op_norm} For every $x \in \U \subseteq \mathbb{S}^n$, we have $\|D\rho_x\|^2 = 1$.\end{lemma}

\begin{proof} Let $\dot{x}\in T_x\U$ be a tangent vector with $\|\dot{x}\| = 1$. Then
$$(D\rho_x\dot{x})^2 = \left(1-\frac{\|x-K(x)\|^2}{4}\right)^{-1}\frac{\langle K(x),\dot{x}\rangle^2}{\|x-K(x)\|^2}.$$

\noindent
This quantity is maximized whenever $\langle K(x),\dot{x}\rangle^2$ does, that is, when $\dot{x}$ is the normalized projection of $K(x)$ on the tangent space $T_x\U$. In other words, we have to compute the tangential component of the vector $K(x)$ on the space $T_x\U$. We have that $x \perp T_x\U$ and $\|x\|=1$, so
$$K(x)^\top = K(x)-K(x)^\perp = K(x)-\langle K(x),x\rangle x.$$

\noindent
Then
\begin{align*}
\|K(x)^\top\|^2 & = \langle K(x)-\langle K(x),x\rangle x,K(x)-\langle K(x),x\rangle x\rangle\\
& = \|K(x)\|^2-2\langle K(x),x\rangle\langle K(x),x\rangle + \langle K(x),x\rangle^2\|x\|^2\\
& = 1-\langle K(x),x\rangle^2.
\end{align*}

\noindent
Hence the unitary tangent vector which maximizes $D\rho_x$ is
$$\dot{x} = \frac{K(x)-\langle K(x),x\rangle x}{(1-\langle K(x),x\rangle^2)^{1/2}}$$

\noindent
and an elementary (yet, tedious) computation shows that $(D\rho(x)\dot{x})^2 = 1$.
\end{proof}

\section{Proof of Theorem \ref{sn}}

Finally we prove the main result in this paper.

\begin{proof}[Proof of Theorem \ref{sn}] According to Proposition \ref{desig}, the smooth self-convexity property is equivalent to
\begin{equation}\label{pteq1}
\|\dot{x}\|^2\|D\rho_x\|^2-(D\rho_x\dot{x})^2-\rho(x)D^2\rho_x(\dot{x},\dot{x}) \geq 0
\end{equation}

\noindent
for every $x\in\U$ and $\dot{x}\in T_x\U$. In lemmas \ref{firstderlemma}, \ref{secondderlemma} and \ref{op_norm} we saw that, if $\M = \S^n$ and $\N$ is any complete $\mathcal{C}^2$ submanifold, then
$$D\rho_x\dot{x} = -\left(1-\frac{\|x-K(x)\|^2}{4}\right)^{-1/2}\frac{\langle K(x),\dot{x}\rangle}{\|x-K(x)\|},$$

\begin{align}\label{pteq2}
D^2\rho_x(\dot{x},\dot{x}) & = \frac{1}{\|x-K(x)\|}\left(1-\frac{\|x-K(x)\|^2}{4}\right)^{-1/2}\cdot\nonumber\\
& \left[\frac{1}{4}\langle K(x),\dot{x}\rangle^2\left(1-\frac{\|x-K(x)\|^2}{4}\right)^{-1}\right.\\
&\left.-\left(\langle DK_x\dot{x},\dot{x}\rangle - \|\dot{x}\|^2\langle K(x),x\rangle + \frac{\langle K(x),\dot{x}\rangle^2}{\|x-K(x)\|^2}\right)\right]
\end{align}

\noindent
and $\|D\rho_x\| = 1$. Fix $x \in \U$ and $\dot{x}\in T_x\U$. If we consider the condition metric for $\S^n$ with $\N$ being a single point, $K(x)$, then the right hand side of \eqref{pteq2} remains equal except for that $\langle DK_x\dot{x},\dot{x}\rangle = 0$ because in this case $K$ is a constant map. In Lemma \ref{lemita} we proved that for $\N$ an arbitrary $\mathcal{C}^2$ submanifold, $\langle DK_x\dot{x},\dot{x}\rangle \geq 0$. Hence the left hand side of \eqref{pteq1} for $\N$ an arbitrary $\mathcal{C}^2$ submanifold is bounded below by the corresponding left hand side for $\N = \{K(x)\}$, and the latter is greater or equal than $0$ by Proposition \ref{punctured_sn_smooth}.
 \end{proof}

\section{Punctured $\H^n$}

In this last section we give a counterexample showing that the smooth self-convexity property does not hold when $\M = \H^n$, the hyperbolic space, and $\N$ is a single point. First note that is enough to give a counterexample for $\mathbb{H}^2$. Indeed, consider the disk model for this punctured $\mathbb{H}^n$, $\M = \mathbb{D}^n \setminus \{0\} = \{x \in \mathbb{R}^n\ |\ \|x\|^2 < 1\}\setminus\{0\}$ together with the condition metric given by the (hyperbolic) distance to the origin $\N = \{0\}$. Then the punctured $\mathbb{H}^2$, $\M_2 = \mathbb{D}^2\setminus \{0\}$, can be viewed as a $2$-dimensional submanifold of $\M$. Now, since there is an isometry of $\M$ that fixes every point in $\M_2$, every geodesic segment in $\M_2$ such that its \eqref{log_conv} function is not convex is a geodesic segment in $\M$ such that its \eqref{log_conv} function is not convex. Some geodesic segments in the punctured disk model for $\mathbb{H}^2$ are represented in Figure \ref{fig:disk}. As we can see, its corresponding \eqref{log_conv} functions are not convex.

\begin{figure}[h]
\begin{subfigure}{.45\textwidth}
  \centering
  \includegraphics[width=1\textwidth]{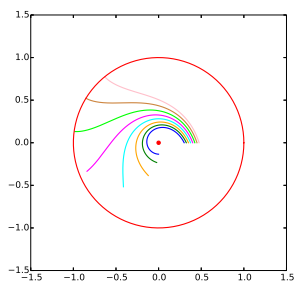}
	\caption{Geodesic segments.}
\end{subfigure}
\begin{subfigure}{.45\textwidth}
  \centering
  \includegraphics[width=1.1\textwidth]{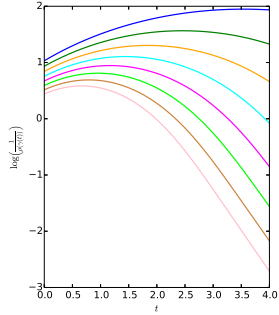}
	\caption{Corresponding \eqref{log_conv} functions.}
\end{subfigure}
\caption{Some geodesic segments in the condition metric when $\M$ is the disk model of the hyperbolic plane $\mathbb{H}^2$ and $\N$ is the red point, $(0,0)$. Clearly the self-convexity property is not satisfied in this case.}\label{fig:disk}
\end{figure}

\begin{proof}[Proof of Theorem \ref{hn}] Let $\mathbb{D}^2$ be the Poincar\'e disk model for the hyperbolic space, $\mathbb{D}^2 = \{z\in\mathbb{C}\ |\ |z|<1\}$. We take polar coordinates $(r,\phi) \mapsto r e^{i\phi}$ with $r \in (0,1)$ and $\phi\in(-\pi,\pi)$. Then the local expression for the metric tensor is
$$g_{(r,\phi)} = \begin{pmatrix} \frac{1}{(1-r)^2} & 0\\
0 & \frac{r^2}{(1-r)^2}\end{pmatrix}.$$

If we take $\N = \{(0,0)\}$, then the (hyperbolic) distance from a point $(r,\phi)$ to $\N$ is $\rho(r,\phi) = -\log(1-r)$. If $(\dot{r},\dot{\phi})$ is a tangent vector at the point $(r,\phi)$, then its norm is given by
\begin{equation}\label{disk_norm}
\|(\dot{r},\dot{\phi})\|^2_{(r,\phi)} = \frac{\dot{r}^2+\dot{\phi}^2r^2}{(1-r)^2}.
\end{equation}

Now let us compute the Christoffel symbols for the Pincar\'e disk. We have that
$$\frac{\partial g_{11}}{\partial_r}=\frac{2}{(1-r)^3} \qquad \frac{\partial g_{22}}{\partial r} = \frac{2r}{(1-r)^3},$$

\noindent
and the rest of the derivatives are zero. The Christoffel symbols are
$$\Gamma_{11}^1 = \frac{1}{1-r}, \qquad \Gamma_{12}^1 = 0, \qquad \Gamma_{22}^1 = -\frac{r}{1-r},$$
$$\Gamma_{11}^2 = 0, \qquad \Gamma_{12}^2 = \frac{1}{r(1-r)}, \qquad \Gamma_{22}^2 = 0.$$

\noindent
With the Christoffel symbols we obtain the geodesic equations
\begin{equation}\label{geod_disk}
\left\{\begin{array}{rl}
\ddot{r} + \frac{\dot{r}^2}{1-r} - \frac{r\dot{\phi}^2}{1-r} & = 0\\
\ddot{\phi} + \frac{2\dot{r}\dot{\phi}}{r(1-r)} = 0
\end{array}
\right.
\end{equation}

Now let us compute the derivatives of the distance function $\rho$. Let $(r,\phi)$ be a point and $(\dot{r},\dot{\phi})$ a tangent vector. Let $c(t) = (c_1(t),c_2(t))$ be a curve with $c(0) = (r,\phi)$ and $c'(0) = (\dot{r},\dot{\phi})$. We have that
$$\frac{d}{dt}\rho(c(t)) = \frac{d}{dt}[-\log(1-c_1(t))] = \frac{c_1'(t)}{1-c_1(t)}.$$

\noindent
Hence,
\begin{equation}\label{first_der_disk}
D\rho_{(r,\phi)}(\dot{r},\dot{\phi}) = \frac{\dot{r}}{1-r}.
\end{equation}

Now let $\gamma(t) = (\gamma_1(t),\gamma_2(t))$ be a geodesic (w.r.t. the original hyperbolic metric) with $\gamma(0) = (r,\phi)$ and $\gamma'(0) = (\dot{r},\dot{\phi})$. Then,
$$\frac{d^2}{dt^2}\rho(\gamma(t)) = \frac{d}{dt}\frac{\gamma_1'(t)}{1-\gamma_1(t)} = \frac{\gamma_1''(t)(1-\gamma_1(t))+\gamma_1'(t)^2}{(1-\gamma_1(t))^2}.$$

\noindent
Therefore,
\begin{equation}\label{second_der_disk}
D^2\rho_{(r,\phi)}((\dot{r},\dot{\phi}),(\dot{r},\dot{\phi})) = \frac{\ddot{r}(1-r)+\dot{r}^2}{(1-r)^2} = \frac{r\dot{\phi}^2}{(1-r)^2},
\end{equation}

\noindent
where we have replaced $\ddot{r}$ by its value in terms of $\dot{r}$ and $\dot{\phi}$ using the geodesic equations \eqref{geod_disk}. Let us compute the operator norm of $D\rho_{(r,\phi)}$. The quantity in \eqref{first_der_disk} is maximized when $\dot{r}$ is as large as possible. Let us consider the tangent vector $(1,0)$, whose norm is $\frac{1}{1-r}$. Then $(1-r,0)$ is a unitary vector that maximizes $D\rho_{(r,\phi)}$. Hence,
\begin{equation}\label{op_norm_disk}
\|D\rho_{(r,\phi)}\| = D\rho_{(r,\phi)}(1-r,0) = 1.
\end{equation}

\noindent
Finally, let us compute quantity in Proposition \ref{desig} using \eqref{disk_norm}, \eqref{first_der_disk}, \eqref{second_der_disk} and \eqref{op_norm_disk}.
\begin{align*}
\|(\dot{r},\dot{\phi})\|^2\|D\rho_{(r,\phi)}\|^2 - (D\rho_{(r,\phi)}(\dot{r},\dot{\phi}))^2 &\\
-\rho(r,\phi)D^2\rho_{(r,\phi)}((\dot{r},\dot{\phi}),(\dot{r},\dot{\phi})) & = \frac{\dot{\phi}^2r(r+\log(1-r))}{(1-r)^2}.
\end{align*}

\noindent
Since the real function $r \mapsto r+\log(1-r) < 0$ for every $r \in (0,1)$, the quantity above is zero if and only if $\dot{\phi} = 0$ ($(\dot{r},\dot{\phi})$ points towards the origin) and otherwise is negative. Proposition \ref{desig} finishes the proof.
\end{proof}

\end{document}